\documentclass[12pt]{article}

\usepackage{amsthm,amsmath,amssymb}
\usepackage{enumerate}
\usepackage[colorlinks=true,citecolor=black,linkcolor=black,urlcolor=blue]{hyperref}
\usepackage{dsfont}
\usepackage{bm}
\usepackage[colorlinks=true,citecolor=black,linkcolor=black,urlcolor=blue]{hyperref}
\usepackage{graphicx}
\usepackage{subfigure}
\usepackage{tikz}
\usepackage{authblk}
\def\multiset#1#2{\ensuremath{\left(\kern-.3em\left(\genfrac{}{}{0pt}{}{#1}{#2}\right)\kern-.3em\right)}}
\allowdisplaybreaks

\numberwithin{equation}{section}

\theoremstyle{plain}
\newtheorem{theorem}{Theorem}[section]
\newtheorem{lemma}[theorem]{Lemma}
\newtheorem{corollary}[theorem]{Corollary}

\title{\bf Counting generalized Schr\"{o}der paths}

\author{Xiaomei Chen\thanks{Corresponding author: xmchen@hnust.edu.cn}, Yuan Xiang}
\affil{\small School of Mathematics and Computational Science\authorcr \small Hunan University of Science and Technology\authorcr \small Xiangtan 411201, China}
\date{\small Mathematics Subject Classifications: 05A15}

\begin{document}

\maketitle

\begin{abstract}
A Schr\"{o}der path is a lattice path from $(0,0)$ to $(2n,0)$ with steps $(1,1)$, $(1,-1)$ and $(2,0)$ that never goes below the $x-$axis. A small Schr\"{o}der path is a Schr\"{o}der path with no $(2,0)$ steps on the $x-$axis. In this paper, a 3-variable generating function $R_L(x,y,z)$ is given for  Schr\"{o}der paths and small Schr\"{o}der paths respectively. As corollaries, we obtain the generating functions for several kinds of generalized Schr\"{o}der paths counted according to the order in a unified way.

\end{abstract}

\section{Introduction}
In this paper, we will consider the following sets of steps for lattice paths:
\begin{align*}
	S_1&=\{(1,1),(1,-1)\},\\
	S_2&=\{(r,r),(r,-r)|r\in \mathbb{N}^{+}\},\\
    S_3&=\{(1,1),(1,-1),(2,0)\},\\
    S_4&=\{(r,r),(r,-r),(2r,0)|r\in \mathbb{N}^{+}\},\\
    S_5&=\{(1,1),(1,-1),(2r,0)|r\in \mathbb{N}^{+}\},\\
    S_6&=\{(r,r),(r,-r),(2,0)|r\in \mathbb{N}^{+}\},
\end{align*}
where $(r,r)$, $(r,-r)$, $(2r,0)$ are called \emph{up steps}, \emph{down steps} and \emph{horizontal steps} respectively.

For a given set $S$ of steps, let $L_S(n)$ denote the set of lattice paths from $(0,0)$ to $(2n,0)$ with steps in $S$, and never go below the $x-$axis. Let $A_S(n)$ denote the subset of $L_S(n)$ whose member paths have no horizontal steps on the $x-$axis. We denote by $L_S=\bigcup_{n\geq 1}L_S(n)$ and $A_S=\bigcup_{n\geq 1}A_S(n)$. Then $L_{S_1}(n)$, $L_{S_3}(n)$ and $A_{S_3}(n)$ are the sets of \emph{Dyck paths}, \emph{Schr\"{o}der paths} and \emph{small Schr\"{o}der paths} of \emph{order} $n$ respectively.

It is well known that $|L_{S_1}(n)|$ is the $n$th \emph{Catalan number} (A000108 in \cite{OEIS}), $|L_{S_3}(n)|$ is the $n$th \emph{large Schr\"{o}der number} (A006318), and $|A_{S_3}(n)|$ is the $n$th \emph{small Schr\"{o}der number} (A001003). Define a \emph{peak} in a Dyck path to be a vertex between an up step and a down step. Then the number of Dyck paths of order $n$ with $k$ peaks is the well known \emph{Narayana number} (A001263)
$$N(n,k)=\frac{1}{n}\binom{n}{k}\binom{n}{k-1}.$$
The $n$th \emph{Narayana polynomial} is defined as $N_n(y)=\sum_{1\leq k\leq n}N(n,k)y^k$ for $n\geq 1$ with $N_0(y)=1$. In \cite{Sulanke1}, Sulanke gave the generating function for the Narayana polynomial as
\begin{equation}\label{equ:Sulanke1}
  \sum_{n\geq 0}N_n(y)x^n=(1+(1-y)x-\sqrt{1-2(1+y)x+(1-y)^2x^2})/(2x).
\end{equation}

Let
$$P_{S_i}(x)=1+\sum_{n\geq 1}|L_{S_i}(n)|x^n$$
and
$$Q_{S_i}(x)=1+\sum_{n\geq 1}|A_{S_i}(n)|x^n$$
denote the generating functions for $|L_{S_i}(n)|$ and $|A_{S_i}(n)|$ respectively. As one type of generalization of Dyck paths, $L_{S_2}(n)$ has been studied by several authors. The generating function $P_{S_2}(x)$ is given in \cite{Kung} and \cite{Coker} with different methods as
\begin{equation}\label{equ:Kung1}
  P_{S_2}(x)=(1+3x-\sqrt{1-10x+9x^2})/(8x).
\end{equation}
Moreover, Coker \cite{Coker} and Sulanke \cite{Sulanke1} expressed $|L_{S_2}(n)|$ as a combination of Narayana numbers, and Woan \cite{Woan} gave a three-term recurrence for $|L_{S_2}(n)|$. For other types of generalization of Dyck paths, readers can refer to \cite{Huq} and \cite{Rukavicka}.

Comparing to the above results about generalization of Dyck paths, generalization of Schr\"{o}der paths has been rarely studied until Kung and Miler \cite{Kung} gave the generating functions $P_{S_i}(x)$ ($4\leq i\leq 6$). Later, Huh and Park \cite{Huh} expressed $|A_{S_4}(n)|$ as a combination of Narayana numbers.

Note that we can also obtain Equation \eqref{equ:Kung1} by considering the number of runs of Dyck paths. Here a \emph{run} in a lattice path is defined to be a vertex between two consecutive steps of the same kind. Let $R(n,k,S_1)$ denote the number of lattice paths in $L_{S_1}(n)$ with $k$ runs. Since a Dyck path of order $n$ with $k$ peaks has $2n-2k$ runs, we obtain from Equation \eqref{equ:Sulanke1} that
\begin{equation}\label{equ:Sulanke2}
\begin{aligned}
  &1+\sum_{n,k\geq 1}R(n,k,S_1)x^ny^k=1+\sum_{n,k\geq 1}N(n,k)x^ny^{2n-2k}\\
  =&(1+(y^2-1)x-\sqrt{1-2(1+y^2)x+(1-y^2)^2x^2})/(2xy^2).
  \end{aligned}
\end{equation}
Then Equation \eqref{equ:Kung1} is derived from Equation \eqref{equ:Sulanke2} by setting $y=2$.

Motivated by the above observation, we study the number of runs for Schr\"{o}der paths according to the following two types: a run is \emph{diagonal} if it is the joint of two up steps or two down steps, and a run is \emph{horizontal} if it is the joint of two horizontal steps.

For a Schr\"{o}der path $P$, let $\mathrm{dr}(P)$, $\mathrm{hr}(P)$ and $\mathrm{order}(P)$ denote the number of diagonal runs, the number of horizontal runs and the order of $P$ respectively. Then the generating function $R_L(x,y,z)$ is defined for $L\subseteq L_{S_3}$ as
\begin{equation*}
R_L(x,y,z)=1+\sum_{P\in L}x^{\mathrm{order}(P)}y^{\mathrm{dr}(P)}z^{\mathrm{hr}(P)}.
\end{equation*}
In this paper, we give $R_L(x,y,z)$ for $L=L_{S_3}$ and $L=A_{S_3}$. As corollaries, we obtain the generating functions $P_{S_i}(x)$ and $Q_{S_i}(x)$ for $4\leq i\leq 6$ in a unified way.

\section{The case for Schr\"{o}der paths}
In the following, we use $U$, $D$ and $H$ to denote the steps $(1,1)$, $(1,-1)$ and $(2,0)$ respectively. For a lattice $P$ and a step $s$, the \emph{insertion} of $s$ at a vertex $v$ of $P$ is defined as following: decompose $P$ into two parts at $v$ as $P=P_1P_2$, where $P_i$ maybe empty. Then we connect the initial vertex of $s$ to the end vertex of $P_1$, and connect the end vertex of $s$ to the initial vertex of $P_2$. See Figure \ref{fig:Insertion} for an example.
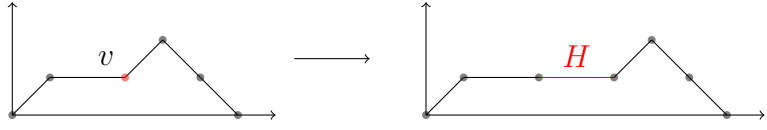
\begin{figure}[t]
\centering
\begin{tikzpicture}[scale=0.5]
\draw[black] (0,0) -- (1,1);\draw[black] (1,1) -- (3,1);\draw[black] (3,1) -- (4,2);
\draw[black] (4,2) -- (5,1);\draw[black] (5,1) -- (6,0);
\draw[->] (0,0)-- (7,0); \draw[->] (0,0) -- (0,3);
\fill [black,opacity=.5](0,0)circle (3pt);
\fill [black,opacity=.5](1,1)circle (3pt);
\fill [red,opacity=.5](3,1)circle (3pt);
\fill [black,opacity=.5](4,2)circle (3pt);
\fill [black,opacity=.5](5,1)circle (3pt);
\fill [black,opacity=.5](6,0)circle (3pt);
\node [black,left] at (3,1.5) {$v$};

\path[->] (7.5,1.5) edge (9.5,1.5);

\draw[black] (11,0) -- (12,1);\draw[black] (12,1) -- (14,1);\draw[red] (14,1) --node [above] {$H$} (16,1);\draw[black] (16,1) -- (17,2);
\draw[black] (17,2) -- (18,1);\draw[black] (18,1) -- (19,0);
\draw[->] (11,0)-- (20,0); \draw[->] (11,0) -- (11,3);
\fill [black,opacity=.5](11,0)circle (3pt);
\fill [black,opacity=.5](12,1)circle (3pt);
\fill [black,opacity=.5](14,1)circle (3pt);
\fill [black,opacity=.5](16,1)circle (3pt);
\fill [black,opacity=.5](17,2)circle (3pt);
\fill [black,opacity=.5](18,1)circle (3pt);
\fill [black,opacity=.5](19,0)circle (3pt);
\end{tikzpicture}
\caption{An example of the insertion of an $H$ step.}
\label{fig:Insertion} 
\end{figure}

Given $P\in L_{S_1}(n)$ with $k$ peaks, let $V$ denote the set of vertices of $P$ other than runs. We then insert $m$ $H$ steps to $P$ as following:
\begin{enumerate}[(1)]
  \item We firstly choose $i$ vertices from $V$, and insert an $H$ step at each chosen vertex. In this step, we have $\binom{2k+1}{i}$ choices, and each insertion has no effect to the number of runs.
  \item For the lattice path obtained after step (1), we choose $j$ vertices from its runs, and insert an $H$ step at each chosen vertex. In this step, we have $\binom{2n-2k}{j}$ choices, and the number of diagonal runs will decrease by $j$ after insertion.
  \item For the lattice path obtained after step (2), we insert the remaining $m-i-j$ $H$ steps immediately after the $i+j$ $H$ steps that have been inserted. In this step, we have $\multiset{m-i-j}{i+j}$ choices, and the number of horizontal runs will increase by $m-i-j$ after insertion.
\end{enumerate}

Let $\mathrm{Ins}_m(P)$ denote the set of all Schr\"{o}der paths obtained from $P$ by the above insertion. Then we have
\begin{equation}\label{equ:insert}
  |\mathrm{Ins}_m(P)|=\binom{2k+1}{i}\binom{2n-2k}{j}\multiset{m-i-j}{i+j}.
\end{equation}
Moreover, we have $\mathrm{order}(P')=m+n$, $\mathrm{dr}(P')=2n-2k-j$, $\mathrm{hr}(P')=m-i-j$ for each $P'\in \mathrm{Ins}_m(P)$.

On the other hand, let $HL_{S_3}$ denote the subset of $L_{S_3}$ whose member paths consisting of $H$ steps only. Let $UL_{S_3}$ denote the subset of $L_{S_3}$ whose member paths have at least one $U$ step. It is obvious that each path of $UL_{S_3}$ can be obtained uniquely from a Dyck path by inserting some $H$ steps as above. Thus we have
$$L_{S_3}=HL_{S_3}\cup UL_{S_3}=HL_{S_3}\cup \bigcup_{P\in L_{S_1}, m\geq 0}\mathrm{Ins}_m(P).$$
Summarizing the above discussion, we then obtain the following result.
\begin{theorem}\label{thm:schroder}
\begin{equation*}
\begin{aligned}
  &R_{L_{S_3}}(x,y,z)\\
  =&\frac{1-xz+x}{1-xz}(1+(1-w)u-\sqrt{1-2(1+w)u+(1-w)^2u^2})/(2u),
\end{aligned}
\end{equation*}
where $u=x\left(\frac{x+y(1-xz)}{1-xz}\right)^2$ and $w=\left(\frac{1-xz+x}{x+y(1-xz)}\right)^2$.
\end{theorem}
\begin{proof}
By Equation \eqref{equ:insert}, we have
\begin{small}
\begin{align*}
  &R_{L_{S_3}}(x,y,z)\\
  =&1+\sum_{P\in HL_{S_3}}x^{\mathrm{order}(P)}y^{\mathrm{dr}(P)}z^{\mathrm{hr}(P)}+\sum_{P\in UL_{S_3}}x^{\mathrm{order}(P)}y^{\mathrm{dr}(P)}z^{\mathrm{hr}(P)}\\
  =&1+\frac{x}{1-xz}\\
  &+\sum_{\substack{n,k\geq 1\\m,i,j\geq 0}}N(n,k)\binom{2k+1}{i}\binom{2n-2k}{j}\multiset{m-i-j}{i+j}x^{m+n}y^{2n-2k-j}z^{m-i-j}\\
  =&\frac{1-xz+x}{1-xz}+\sum_{n,k\geq 1}N(n,k)x^ny^{2n-2k}\\
  &\cdot\sum_{\substack{0\leq i\leq 2k+1\\0\leq j\leq 2n-2k}}\binom{2k+1}{i}\binom{2n-2k}{j}x^{i+j}y^{-j}\sum_{m\geq i+j}\binom{-i-j}{m-i-j}(-xz)^{m-i-j}\\
  =&\frac{1-xz+x}{1-xz}\\
  &+\sum_{n,k\geq 1}N(n,k)x^ny^{2n-2k}\sum_{\substack{0\leq i\leq 2k+1\\0\leq j\leq 2n-2k}}\binom{2k+1}{i}\binom{2n-2k}{j}x^{i+j}y^{-j}(1-xz)^{-i-j}\\
  =&\frac{1-xz+x}{1-xz}+\sum_{n,k\geq 1}N(n,k)x^ny^{2n-2k}\left(1+\frac{x}{y(1-xz)}\right)^{2n-2k}\left(1+\frac{x}{1-xz}\right)^{2k+1}\\
  =&\frac{1-xz+x}{1-xz}\sum_{n\geq 0}N_{n}(w)u^n,
\end{align*}
\end{small}
then Theorem \ref{thm:schroder} is derived from Equation \eqref{equ:Sulanke1}.
\end{proof}

The generating functions $P_{S_i}(x)$ for $4\leq i\leq 6$ were derived by Kung and Mier \cite{Kung}. Here we can obtain them as a direct corollary of the above result.
\begin{corollary}\cite{Kung}\label{coro:Kung}
\begin{align*}
    &P_{S_4}(x)=\frac{(1-x)(1+x-4x^2)-(1-x)^2\sqrt{1-12x+16x^2}}{2x(2-3x)^2},\\
    &P_{S_5}(x)=\frac{2x-1+\sqrt{1-8x+12x^2-4x^3}}{2x(x-1)},\\
    &P_{S_6}(x)=\frac{1+2x-x^2-\sqrt{(1-x)(1-11x+7x^2-x^3)}}{2x(x-2)^2}.
\end{align*}
\end{corollary}
\begin{proof}
We use a bijection given by Huh and Park \cite{Huh}. Let $\bar{L}_{S_3}(n)$ denote the set of Schr\"{o}der paths of order $n$ whose runs are colored in either black or white, and other vertices are colored in black only. For $P\in \bar{L}_{S_3}(n)$. Let $\phi(P)$ denote the lattice path obtained from $P$ as following: delete all white vertices of $P$, and then connect adjacent black vertices with line segments. See \cite[Figure 8]{Huh} for an example. It is obvious that $\phi$ is a bijection from $\bar{L}_{S_3}(n)$ to $L_{S_4}(n)$, which implies that
$$P_{S_4}(x)=R_{L_{S_3}}(x,2,2)=\frac{(1-x)(1+x-4x^2)-(1-x)^2\sqrt{1-12x+16x^2}}{2x(2-3x)^2}.$$
Similarly, we can obtain $P_{S_5}(x)$ and $P_{S_6}(x)$ by setting the pair $(y,z)$ to be $(1,2)$ and $(2,1)$ in $R_{L_{S_3}}(x,y,z)$ respectively.
\end{proof}

Using the techniques in \cite{Flajolet} ([Chapter VI]), Kung and Miler \cite{Kung} gave the asymptotic formula for $|L_{S_4}(n)|$. The asymptotic formulas for $|L_{S_5}(n)|$ and $|L_{S_6}(n)|$ can be obtained from  Corollary \ref{coro:Kung} in a similar way:
\begin{align*}
  |L_{S_4}(n)|\sim &\frac{\beta_1}{\alpha_1^n\sqrt{\pi n^3}},\\
  |L_{S_5}(n)|\sim &\frac{\beta_2}{\alpha_2^n\sqrt{\pi n^3}},\\
  |L_{S_6}(n)|\sim &\frac{\beta_3}{\alpha_3^n\sqrt{\pi n^3}},
\end{align*}
where $\alpha_i$ and $\beta_i$ are defined as following:
\begin{enumerate}[(1)]
  \item $\alpha_1=\frac{3-\sqrt{5}}{8}$ is the root of equation $f_1(x)=1-12x+16x^2=0$, and $\beta_1=\frac{(1-\alpha_1)^2\sqrt{-\alpha_1f_1'(\alpha_1)}}{4\alpha_1(2-3\alpha_1)^2}=\frac{(35-15\sqrt{5})(\sqrt{6\sqrt{5}-10})}{4}$;
  \item $\alpha_2=0.16243\cdots$ is the root of equation $f_2(x)=1-8x+12x^2-4x^3=0$, and $\beta_2=\frac{\sqrt{-\alpha_2f_2'(\alpha_2)}}{4\alpha_2(1-\alpha_2)}=1.55669\cdots$;
  \item $\alpha_3=0.09678\cdots$ is the root of equation $f_3(x)=(1-x)(1-11x+7x^2-x^3)$, and $\beta_3=\frac{\sqrt{-\alpha_3f_3'(\alpha_3)}}{4\alpha_3(2-\alpha_3)^2}=0.68998\cdots$.
\end{enumerate}

Theorem \ref{thm:schroder} can also be used to study colored Schr\"{o}der paths. For instance, let $a(n)$ denote the number of Schr\"{o}der paths of order $n$ with their horizontal runs colored in one of three given colors. Then we obtain from Theorem \ref{thm:schroder} that
\begin{equation}\label{equ:pyramid}
1+\sum_{n\geq 1}a(n)x^n=R_{L_{S_3}}(x,1,3)=\frac{3x-1+\sqrt{1-10x+25x^2-16x^3}}{2x(2x-1)}.
\end{equation}
The coefficients of the above function appear as sequence A186338 in OEIS, and is related to sequence A091866.

For the definition of pyramid and pyramid weight, see \cite[Definition 2.1]{Denise}. Let $T(n,k)$ denote the number of Dyck paths of order $n$ that have pyramid weight $k$. Combing Equation \ref{equ:pyramid} with a result of Denise and Simion (\cite[Theorem 2.3]{Denise}), we then obtain the following result.
\begin{corollary}
$$a(n)=\sum_{k=0}^n T(n,k)2^k.$$
\end{corollary}

\section{The case for small Schr\"{o}der paths}
A lattice path in $A_{S_3}$ is said to be \emph{primitive} if it does not intersect the $x-$axis except at $(0,0)$ and $(2n,0)$. Let $PA_{S_3}$ denote the set of all primitive paths in $A_{S_3}$. Since every path in $A_{S_3}$ can be decomposed uniquely into a sequence of paths in $PA_{S_3}$, we have
\begin{equation}\label{equ:decompose}
  R_{A_{S_3}}(x,y,z)=\frac{1}{1-\bar{R}_{PA_{S_3}}(x,y,z)},
\end{equation}
where we use $\bar{R}_{L}(x,y,z)$ to denote the function $R_L(x,y,z)-1$ for a given set $L$ of lattice paths.

We now consider the generating function $\bar{R}_{PA_{S_3}}(x,y,z)$. Note that the set $UL_{S_3}$ can be partitioned as $UL_{S_3}=\bigcup_{i=1}^4U_i$, where
\begin{enumerate}[(1)]
  \item $U_1$=\{$P\mid$ $P$ starts with $U$ and ends with $D$\};
  \item $U_2$=\{$P\mid$ $P$ starts with $H$ and ends with $D$\};
  \item $U_3$=\{$P\mid$ $P$ starts with $U$ and ends with $H$\};
  \item $U_4$=\{$P\mid$ $P$ starts and ends with $H$\}.
\end{enumerate}

As shown in Section 2, each path $P\in UL_{S_3}$ can be obtained uniquely from a Dyck path $P'$
by inserting some $H$ steps, and we have the following fact:
\begin{enumerate}[(1)]
  \item if it is not allowed to insert at either the initial vertex or the end vertex of $P'$, then $P\in U_1$;
  \item if it is required to insert at the initial vertex of $P'$, and not allowed to insert at the end vertex, then $P\in U_2$;
  \item if it is required to insert at the end vertex of $P'$, and not allowed to insert at the initial vertex, then $P\in U_3$;
  \item if it is required to insert at both the initial vertex and the end vertex of $P'$, then $P\in U_4$.
\end{enumerate}

Based on the above observation, we can obtain the following result after some calculation.
\begin{lemma}\label{lem:Ai}
\begin{align*}
&\bar{R}_{U_1}(x,y,z)=\frac{(1-xz)^2}{(1-xz+x)^2}\bar{R}_{UL_{S_3}}(x,y,z),\\
&\bar{R}_{U_2}(x,y,z)=\bar{R}_{U_3}(x,y,z)=\frac{x(1-xz)}{(1-xz+x)^2}\bar{R}_{UL_{S_3}}(x,y,z),\\
&\bar{R}_{U_4}(x,y,z)=\frac{x^2}{(1-xz+x)^2}\bar{R}_{UL_{S_3}}(x,y,z).
\end{align*}
\end{lemma}
\begin{proof}
The proof of the above result is almost the same as that of Theorem \ref{thm:schroder}. Here we take $\bar{R}_{U_2}(x,y,z)$ as an example. By the definition of $U_2$ , we have
\begin{small}
\begin{equation*}
\begin{aligned}
  &\bar{R}_{U_2}(x,y,z)=\\
  =&\sum_{\substack{n,k\geq 1\\i,j\geq 0\\m\geq 1}}N(n,k)\binom{2k-1}{i}\binom{2n-2k}{j}\multiset{m-i-j-1}{i+j+1}x^{m+n}y^{2n-2k-j}z^{m-i-j-1}\\
  =&\sum_{n,k\geq 1}N(n,k)x^ny^{2n-2k}\left(1+\frac{x}{y(1-xz)}\right)^{2n-2k}\left(1+\frac{x}{1-xz}\right)^{2k-1}\frac{x}{1-xz}\\
  =&\frac{x(1-xz)}{(1-xz+x)^2}\bar{R}_{UL_{S_3}}(x,y,z).
\end{aligned}
\end{equation*}
\end{small}
\end{proof}

Now we can obtain $R_{A_3}(x,y,z)$ as a direct corollary of the above result.
\begin{theorem}\label{thm:small}
\begin{equation*}
\begin{aligned}
  &R_{A_{S_3}}(x,y,z)\\
  =&\frac{1}{2}+\frac{-1+zx+(1-z)x^2+(1-xz)\sqrt{1-2(1+w)u+(1-w)^2u^2}}{2(y^2z-2y+1)x^2-2y^2x-2x},
  \end{aligned}
\end{equation*}
where $u=x\left(\frac{x+y(1-xz)}{1-xz}\right)^2$ and $w=\left(\frac{1-xz+x}{x+y(1-xz)}\right)^2$.
\end{theorem}
\begin{proof}
Since
$$PA_{S_3}=\{UPD\mid P\; \text{is empty or}\; P\in L_{S_3}\},$$
we obtain from Lemma \ref{lem:Ai} that
\begin{equation}\label{equ:A3}
\begin{aligned}
  &\bar{R}_{PA_{S_3}}(x,y,z)\\
  =&x+x\bar{R}_{HL_{S_3}}+xy^2\bar{R}_{U_{1}}+xy\bar{R}_{U_2}+xy\bar{R}_{U_3}+x\bar{R}_{U_4}\\
  =&\frac{x(1-xz+x)}{1-xz}+\frac{x}{w}\bar{R}_{UL_{S_3}}(x,y,z)\\
  =&\frac{x(1-xz+x)}{1-xz}+\frac{x}{w}(R_{L_{S_3}}(x,y,z)-\frac{1-xz+x}{1-xz}).
  \end{aligned}
\end{equation}
Combining Equation \eqref{equ:decompose}, Equation \eqref{equ:A3} and Theorem \ref{thm:schroder} together, we then obtain Theorem \ref{thm:small}.
\end{proof}

Setting the pair $(y,z)$ to be $(2,2)$, $(1,2)$ and $(2,1)$ respectively in Theorem \ref{thm:small}, we then obtain the generating functions $Q_{S_i}(x)$ for $4\leq i\leq 6$.
\begin{corollary}\label{coro:small1}
  \begin{align*}
    &Q_{S_4}(x)=\frac{1+4x-\sqrt{1-12x+16x^2}}{10x},\\
    &Q_{S_5}(x)=\frac{-1+\sqrt{1-8x+12x^2-4x^3}}{2x(x-2)},\\
    &Q_{S_6}(x)=\frac{-1-4x+x^2+\sqrt{(1-x)(1-11x+7x^2-x^3)}}{2x(x-5)}.
  \end{align*}
\end{corollary}

Expanding the above functions, we have
\begin{align*}
  Q_{S_4}(x)=&1+x+6{x}^{2}+41{x}^{3}+306{x}^{4}+2426{x}^{5}+20076{x}^{6}+\cdots,\\
  Q_{S_5}(x)=&1+x+3{x}^{2}+12{x}^{3}+53{x}^{4}+248{x}^{5}+1209{x}^{6}+\cdots,\\
  Q_{S_6}(x)=&1+x+6{x}^{2}+40{x}^{3}+293{x}^{4}+2286{x}^{5}+18637{x}^{6}+\cdots.
\end{align*}

The coefficients of $Q_{S_4}(x)$ appear as sequence A078009 in OEIS. The generating functions $Q_{S_5}(x)$ and $Q_{S_6}(x)$, to our knowledge, have not been studied before. From Corollary \ref{coro:small1}, we can obtain the following asymptotic formulas:
\begin{align*}
  |A_{S_4}(n)|\sim &\frac{\gamma_1}{\alpha_1^n\sqrt{\pi n^3}},\\
  |A_{S_5}(n)|\sim &\frac{\gamma_2}{\alpha_2^n\sqrt{\pi n^3}},\\
  |A_{S_6}(n)|\sim &\frac{\gamma_3}{\alpha_3^n\sqrt{\pi n^3}},
\end{align*}
where $\alpha_i$ and $f_i$ are the same as those in Section 2, and $\gamma_i$ is defined as following:
\begin{align*}
  \gamma_1&=\frac{\sqrt{-\alpha_1f_1'(\alpha_1)}}{20\alpha_1}=\frac{\sqrt{10+6\sqrt{5}}}{10},\\
  \gamma_2&=\frac{\sqrt{-\alpha_2f_2'(\alpha_2)}}{4\alpha_2(2-\alpha_2)}=0.70954\cdots,\\
  \gamma_3&=\frac{\sqrt{-\alpha_3f_3'(\alpha_3)}}{4\alpha_3(5-\alpha_3)}=0.50971\cdots.
  \end{align*}

It is well known (see, for example, \cite{Deutsch,Stanley}) that $|L_{S_3}(n)|=2|A_{S_3}(n)|$. Comparing the asymptotic formulas of $|L_{S_i}(n)|$ and $|A_{S_i}(n)|$ for $4\leq i\leq 6$, we have the following analogue.
\begin{corollary}
\begin{align*}
   \lim_{n\rightarrow \infty}\frac{|L_{S_4}(n)|}{|A_{S_4}(n)|}&=\frac{5(1-\alpha_1)^2}{(2-3\alpha_1)^2}=1.39320\cdots,\\
   \lim_{n\rightarrow \infty}\frac{|L_{S_5}(n)|}{|A_{S_5}(n)|}&=\frac{2-\alpha_2}{1-\alpha_2}=2.19393\cdots, \\
   \lim_{n\rightarrow \infty}\frac{|L_{S_5}(n)|}{|A_{S_5}(n)|}&=\frac{5-\alpha_3}{(2-\alpha_3)^2}=1.35364\cdots.
\end{align*}
\end{corollary}

By giving a bijection between 5-colored Dyck paths and $|A_{S_4}(n)|$, Huh and Park \cite{Huh} gave the following expression for $|A_{S_4}(n)|$, which we can also prove here with generating function.
\begin{corollary}\label{coro:small2}\cite{Huh}
\begin{equation*}
|A_{S_4}(n)|=\sum_{k=1}^n N(n,k)5^{n-k}.
\end{equation*}
\end{corollary}
\begin{proof}
By Equation \eqref{equ:Sulanke1}, we have
\begin{equation*}
1+\sum_{n,k\geq 1} N(n,k)5^{n-k}x^n=\sum_{n\geq 0}N_n(\frac{1}{5})(5x)^n=\frac{1+4x-\sqrt{1-12x+16x^2}}{10x}.
\end{equation*}
Then Corollary \ref{coro:small2} is derived from Corollary \ref{coro:small1}.
\end{proof}

\section*{Acknowledgements}
The calculations were done with the help of MAPLE.

\end{document}